\documentclass[a4paper]{article}
\usepackage[latin1]{inputenc}
\usepackage{graphicx,difftrees}
\usepackage{geometry}
\usepackage{amssymb, amsmath, amsfonts, amsthm}
\usepackage{url}
\usepackage{hyperref}
\usepackage{color}
\usepackage{euscript}

\newcommand{\leaveout}[1]{\centerline{
\framebox{ Something has been left out here}}}

\newcommand{\forest}{\ensuremath{\bar{T}}}
\newcommand{\rank}{\ensuremath{\mathrm{rank}}}

\newtheorem{theorem}{Theorem}
\newtheorem{lemma}[theorem]{Lemma}

\title{\bf The minimal stage, energy preserving Runge-Kutta method for polynomial Hamiltonian systems is the Averaged Vector Field method}
\author{E. Celledoni and B. Owren and Y. Sun}

\begin{document}
\maketitle

\begin{abstract}
No Runge-Kutta method can be energy preserving for all Hamiltonian systems. 
But for problems in which the Hamiltonian is a polynomial, the Averaged Vector Field (AVF) method can be interpreted as a Runge-Kutta method whose weights $b_i$ and abscissae $c_i$ represent a quadrature rule of degree at least that of the Hamiltonian. We prove that when the number of stages is minimal, the Runge-Kutta scheme must in fact be identical to the AVF scheme.
 \end{abstract}
 
 \section{Introduction and main result}
 We shall be concerned with canonical Hamiltonian systems
 \begin{equation}\label{eq:hamilt}
      y' = J^{-1} \nabla H(y) = f(y),\qquad J=\left(\begin{array}{cc} 0& I\\-I & 0\end{array}\right).
 \end{equation}
 The numerical solution of problems of the this type has been treated extensively in the literature, we refer to the monographs \cite{hairer06gni, leimkuhler04shd} and the references therein for details. Two of the most important properties of the system \eqref{eq:hamilt} are 
 that the flow is a symplectic map and that the Hamiltonian $H(y)$ is preserved along any solution $y(t)$. The circumstances under which various numerical integrators inherit these two properties are by now fairly well understood. The focus in the present paper is the preservation of the Hamiltonian itself, we study integrators generating a sequence of approximations $\{y_n\}$ to the solution of \eqref{eq:hamilt} such that $H(y_n)=H(y_0)$ for all $n\geq 1$.
In particular we consider what can be achieved when the Hamiltonian is polynomial and the integrator is a Runge-Kutta method. For linear Hamiltonians, the resulting ODE is constant and any consistent Runge-Kutta scheme will reproduce the exact solution. If the Hamiltonian is quadratic, then the resulting ODE is linear, and the condition for preserving energy is that the stability function of the method satisfies $R(z)R(-z)=1$. For polynomials of higher order it is not known to which extent Runge-Kutta methods can preserve the Hamiltonian.
 However, it was noted in \cite{quispel08anc} that the Averaged Vector Field (AVF) method, defined as
 \begin{equation}\label{eq:AVFmethod}
      y_{n+1} = y_n + h\int_0^1 f((1-\xi)y_n+\xi y_{n+1})\,\mathrm{d}\xi
 \end{equation}
 preserves the Hamiltonian for all problems of the form \eqref{eq:hamilt}. The AVF method has second order convergence.
 In particular, when the Hamiltonian is a polynomial, the integral can be exactly resolved a priori, the same result is obtained if the integral in \eqref{eq:AVFmethod} is replaced by a quadrature rule of sufficiently high order. This was observed in \cite{celledoni09epr}.
 In fact,  a standard linear quadrature formula with abscissae 
 $c=(c_1,\ldots,c_s)^T$ and weights 
 $b=(b_1,\ldots,b_s)^T$, results in a Runge-Kutta method in which the Butcher matrix  is given as $A=cb^T$. This immediately shows that for any polynomial Hamiltonian system, there exist Runge-Kutta methods which exactly preserve the energy. Note also that any choice of quadrature rule of sufficiently high order yields the same approximation, the AVF method is reproduced exactly.

 As pointed out in
 \cite{celledoni09epr} any energy-preserving integrator for \eqref{eq:hamilt} must obey all quadrature conditions, but for polynomial systems this can be relaxed. Letting the Hamiltonian be a polynomial of degree $m$, a necessary condition for the energy to be preserved is that the quadrature conditions hold up to order $m$, or in terms of Runge-Kutta coefficients
 \begin{equation}\label{eq:quadcond}
      \sum_i b_ic_i^{k-1} = \frac{1}{k},\quad k=1,\ldots, m.
 \end{equation}
 Thus, in considering energy preserving Runge-Kutta methods for polynomial Hamiltonians of degree $\leq m$ one may immediately restrict the focus to schemes whose coefficients satisfy \eqref{eq:quadcond}. If $m=2s$ then the smallest possible number of stages in the scheme is $s$ the resulting abscissae and weights are those of the Gauss-Legendre quadrature rule.
 If $m=2s-1$ then the smallest possible number of stages is still $s$, but the quadrature rule is not uniquely given, although applying the corresponding Runge--Kutta method with $A=cb^T$ yields the same result for all $(c_i,b_i)$ satisfying \eqref{eq:quadcond} for all $k\leq 2s-1$.  In such a situation, we are interested in answering the question of whether the Butcher matrix $A=cb^T$ is unique. We shall restrict our search to Butcher matrices satisfying the usual condition
 \begin{equation} \label{eq:rowsum}
 \sum_{j=1}^s a_{ij} = c_i,\quad i=1,\ldots,s.
 \end{equation}
 We shall prove the following theorem
 \begin{theorem} \label{theo:main}
 Let $m\geq 3$. Among all Runge--Kutta methods which exactly preserve all polynomial Hamiltonians of degree at most $m$, those with the minimal number of stages coincide with the AVF method \eqref{eq:AVFmethod} when applied to such problems.
 The number of stages in these methods is $\lfloor(m+1)/2\rfloor$.
 \end{theorem}
 In general, there are energy preserving Runge-Kutta methods for polynomial Hamiltonian systems which do not coincide with the AVF-integrator. There also exist such methods of arbitrarily high order. Examples are easily obtained as composition methods based on the AVF-integrator, or by the collocation methods proposed in \cite{iavernaro09hos} and \cite{hairer10epc}. In the rest of this paper, we prove Theorem~\ref{theo:main}. The technique we use can be summarized as follows
 \begin{enumerate}
 \item The first step is to consider a set of conditions for energy preservation which are linear in the Butcher matrix $A$, these conditions are called the \emph{double bush conditions} and may be thought of as a linear system $M(A)=w$ where $M$ is a linear map from $\mathbf{R}^{s\times s}$ into $\mathbf{R}^N$ for some $N$ to be specified.
 \item Then the rank of $M$ is determined, the results are different in the even $(m=2s)$ and
 odd $(m=2s-1)$ cases. A particular basis for the kernel of $M$ is identified in each of the cases.
 \item The discretized AVF method, $A=cb^T$, represents a known solution and any other solution must be
 of the form $A=cb^T + N$, where $N$ is in the kernel of $M$. We show, by using certain nonlinear energy preserving conditions that such solutions require $N=0$.
 \end{enumerate}
 
 The rest of the paper is organized as follows: In Section~\ref{sec:prelim} we review some tools needed from the literature, and we provide general conditions for energy preservation of B-series methods needed in the proof.
 In sections three and four, we prove Theorem~\ref{theo:main} for the cases of even and odd polynomial degree of the Hamiltonian respectively.
  
 \section{Energy preservation for B-series methods}\label{sec:prelim}
 In order to make the paper self contained, we begin this section by
  reviewing some general tools from \cite{chartier06aaa}, see also \cite{faou04ecw} and \cite{celledoni10epi}. General
 conditions for energy preservation are derived for integrators
which possess a B-series expansion. These include the Runge-Kutta methods as a subclass.
A good account of order theory and B-series can be found in the monographs
\cite{butcher08nmf,hairer06gni,hairer93sod}, but for completeness we recount briefly the main ingredients we need.

Let $T$ be the set of rooted trees, and for $t\in T$ we write $|t|$ for its number of vertices.
A forest is an unordered finite collection of trees from $T$, $t_1t_2{\ldots}t_q$ were each tree can appear several times, one may then write $\tau=t_1^{r_1}\cdots t_p^{r_p}$ 
for distinct members $t_1,\ldots,t_p$, indicating that $t_i$ appears $r_i$ times. The set of all forests is denoted $\forest$, and the order $|\tau|$ of a forest is the sum of the orders of each of its elements.
An element of $T$ is either the one-node tree $\ab$, or consists of a root to which a forest is attached, we use the notation $t=[t_1,\ldots,t_q]$ or sometimes
 $t=[t_1^{r_1}t_2^{r^2}\ldots t_p^{r_p}]$ so that each distinct tree $t_i$ occurs $r_i$ times as a subtree of $t$. Sometimes we shall write $B_-(t)$ to denote 
 the forest consisting of the subtrees of $t$. 
 The symmetry coefficient is defined as
\[
    \sigma(\ab)=1,\quad \sigma(t)=r_1!\cdots r_p!\,\sigma(t_1)\cdots\sigma(t_p).
\]
We recall that a large class of integrators, including in particular the Runge-Kutta methods, can be formally expanded into 
an infinite series in terms of derivatives of the vector field $f$, indexed by the set of rooted trees $T$. Writing $y_1 = \psi_h(y)$ where $\psi_h$ is the numerical flow map,
we have
\begin{equation}\label{BseriesED}
    y_1=B(a,y) = y + ha(\ab)f(y) + \frac{h^2}{\sigma(\aabb)} a(\aabb)f'(f)(y)+\cdots
    +\frac{h^{|t|}}{\sigma(t)} a(t) F(t)(y) +\cdots 
\end{equation}
Here $a: T\rightarrow\mathbf{R}$ is a method dependent coefficient map, and $F(t)$ is the elementary differential corresponding to $t$, so that e.g. 
$F(\ab)=f$, $F(\aabb)=f'(f)$. The B-series \eqref{BseriesED} can in many cases be conveniently extended to allow for pullback expansions of functions along the B-series map $B(a,y)$, see for instance \cite{chartier06aaa} or for non-commutative structures we refer to \cite{berland05aso,owren99rkm}. It requires the extension of $a(t)$ to forests, setting for any forest $\tau=t_1^{r_1}\dots t_m^{r_m}$, $a(\tau)=a(t_1)^{r_1}\cdots  a(t_m)^{r_m}$. One has
for any real valued smooth function $G$,
\begin{align}
G(B(a,\cdot)) &= G + ha(\ab)\,G'f+h^2a(\ab)^2\, G''(f,f)+ \cdots  \nonumber\\&
= G + \sum_{\tau\in\forest} \frac{h^{|\tau|}}{\sigma(\tau)} a(\tau)
G^{(q)}(F(\tau_1),\ldots,F(\tau_q)) \label{eq:pullback},
\end{align}
where we sum over forests $\tau=\tau_1\tau_2\ldots\tau_q$. If we apply \eqref{eq:pullback} to the special case where $G=H$ and introduce the \emph{elementary Hamiltonian}, defined for every $t=[t_1\ldots t_q]$ as
\[
      H(t) = H^{(q)}(F(t_1),\ldots,F(t_q)),
\]
we conclude that
\begin{equation} \label{eq:HBay}
H(B(a,y)) =  \sum_{t\in T} \frac{h^{|t|-1}}{\sigma(t)}
\left(\prod_{k=1}^q a(t_k)\right)\,
H(t)(y).
\end{equation}
The term indexed by the one-vertex tree $\ab$ is interpreted
as $H^{(0)}(F(\emptyset))(y) := H(y)$. Thus to formally impose that
$H$ is conserved, i.e. $H(y)=H(B(a,y))$ for a map with B-series $B(a,y)$  amounts to requiring that the the right hand side of \eqref{eq:HBay} minus the first term (with $t=\ab$) sums to zero. It is, however not so that the set of functions $\{H(t), t\in T\}$ is linearly independent. It is well-known that the dependency can be described by means of the \emph{Butcher product}, defined between two trees $u=[u_1{\ldots}u_q]\in T$ and $v\in T$ as
\[
   u\circ v = [u_1u_2\ldots u_qv] \in T.
\] 
This product is non-commutative, and satisfies $|u\circ v| = |u|+|v|$. For  elementary Hamiltonians one has
\begin{equation} \label{eq:Hequiv}
      H(u\circ v) = - H(v\circ u)
\end{equation}
for any pair of trees $u, v\in T$. The two trees $u\circ v$ and $v\circ u$ are topologically identical, $v\circ u$ is obtained from $u\circ v$ by shifting the root one position. Conversely, any two trees $t_1$ and $t_2$ which differ only by such a shift of the root can be represented as $t_1=u\circ v$ and $t_2=v\circ u$ for a certain choice of $u$ and $v$.
This shifting of roots induces an equivalence relation on the set of trees by defining two trees to be equivalent if and only if one can be obtained from the other by zero or more root shifts. All trees in the same equivalence class clearly have the same number of vertices, and each equivalence class is called a \emph{free tree}. The set of all free trees is denoted $FT$ and those with precisely $n$ vertices we call $FT^n$. The canonical projection is denoted $\pi: T\rightarrow FT$.
For two trees $u$ and $v$ in the same equivalence class, we define
$\kappa(u,v)$ to be the number of root shifts necessary to obtain $v$ from $u$ and $\kappa(u,u)=0$.
A special role is played by those trees which have a factorization $u\circ u$, then \eqref{eq:Hequiv} implies $H(u\circ u)=0$. Any free tree which contains a member with such a factorization is called superfluous and we note that superfluous trees have an even number of vertices. The set of nonsuperfluous free trees will hereafter be denoted $FT_*$ and $FT_*^n$.
It follows from \eqref{eq:Hequiv} that for two trees $u$ and $v$ in the same equivalence class, $H(u)=(-1)^{\kappa(u,v)}H(v)$ so, disposing of the superfluous trees for which $H(t)=0$, we may rewrite \eqref{eq:HBay} as follows
\begin{equation} \label{eq:HBayfree}
  H(B(a,y))-H(y) = \sum_{n\geq 2}\sum_{\bar{t}\in FT_*^n}
  h^{|t|-1} H(t)
  \sum_{u\in\pi^{-1}(\bar{t})} \frac{(-1)^{\kappa(t,u)}}{\sigma(u)}\prod_{k=1}^q a(u_k) 
\end{equation}
where $t$ is some designated element in the equivalence class $\bar{t}$, and where
each $u\in\pi^{-1}(\bar{t})$ is composed of subtrees as $u=[u_1{\ldots}u_q]$ ($q$ depends on $u$). It is known that the elementary Hamiltonians corresponding to the set of nonsuperfluous free trees are linearly independent, and that leads us to the condition for energy preservation derived by Chartier et al. \cite{chartier06aaa}, saying that the innermost sum must vanish for every free tree. In fact, since the power of $h$ in the above expression is $|\bar{t}|-1$ we need to consider trees in $FT^{n+1}$ to obtain conditions for energy preservation to order $n$
\begin{theorem}\label{theo:2} \cite{chartier06aaa}
A map with B-series $B(a,y)$ preserves energy up to order $n$ if and only if
\begin{equation} \label{eq:theo2}
\sum_{u\in\pi^{-1}(\bar{t})} \frac{(-1)^{\kappa(t,u)}}{\sigma(u)} a(B_-(u)) = 0,\quad\forall
\bar{t}\in \bigcup_{k\leq n+1} FT^k.
\end{equation}
Here $t$ is a designated member of the equivalence class $\bar{t}$.
\end{theorem}

One may remark that all the conditions of the theorem must be satisfied in order for the corresponding method to be energy preserving for every Hamiltonian function $H$.
However, in this article we are interested in the subclass of Hamiltonians which are multivariate polynomials of some prescribed degree $m$. For such $H(y)$ one realizes that
$H(t)=0$ whenever $t$ contains a vertex with more than $m$ emanating branches. On the other hand, one may verify that all $H(t)$ corresponding to nonsuperfluous free trees with at most $m$ emanating branches from any vertex form a linearly independent set when considered uniformly over all Hamiltonians of degree at most $m$.
The following result is inspired by \cite{celledoni09epr}.
\begin{theorem}
Any consistent B-series method with coefficients $a(t)$ which is energy preserving for all polynomial Hamiltonians of degree $m$ satisfies the quadrature conditions of order $k$ for $1\leq k \leq m$, i.e.
\[
     a([\ab^{k-1}]) = \frac{1}{k},\quad 1\leq k\leq m.
\]
If the method is a Runge-Kutta method with abscissae $c_i$ and weights $b_i$, $i=1,\ldots,s$ a necessary condition for energy preservation is 
\begin{equation} \label{eq:quadRK}
    \sum_{i=1}^s b_ic_i^{k-1} = \frac{1}{k},\quad k=1,\ldots,m.
\end{equation}
\end{theorem}

\begin{proof}
For $1\leq k\leq m$ we consider \eqref{eq:theo2} for the free tree with $k+1$ vertices containing the bushy tree $t=[\ab^{k}]$ (i.e. the tree consisting of $k$ copies of the one-node tree as subtrees).
There is only one other tree in the equivalence class, namely $t'=[[\ab^{k-1}]]$.
Now $a(B_-(t))=a(\ab^k)=a(\ab)^k=1$ for any consistent method. On the other hand
$a(B_-(t'))=a([\ab]^{k-1})$ and together with \eqref{eq:theo2} and the fact that
$\kappa(t,t')=1$, $\sigma(t)=k!$ and $\sigma(t')=(k-1)!$ we get the desired result. For Runge-Kutta methods it is well-known that the $a([\ab^{k-1}])$ is  the left hand side of \eqref{eq:quadRK}.
 \end{proof}
 \subsection{The double bush conditions}
A certain subset of the nonsuperfluous free trees will play a particular role here, these are the trees which yield linear conditions on the matrix $A$.
We consider the \emph{double bush free trees} that we denote
$t_{p,q}$ for integers $p$ and $q$  in $\{1,2,\ldots,m-1\}$. Clearly $t_{p,p}$ is superfluous, and by symmetry, $t_{p,q}=t_{q,p}$, so one will typically require $1\leq p<q\leq m-1$.
\tikzstyle dbtree=[sibling distance=6mm,level distance=6mm,thick]
\tikzstyle dbtree node=[scale=0.7,shape=circle,very thin,draw]
\tikzstyle dbtree black node=[style=dbtree node,fill=black]
\tikzstyle dbtree white node=[style=dbtree node,fill=white]
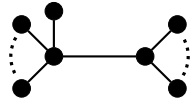
\begin{figure}[ht]
\begin{center}
  \begin{tikzpicture}[dbtree]
    \node[dbtree black node] {}
    child[grow=up] { node[dbtree black node] {} }
    child[grow=north west] { node[dbtree black node] {}}
     child[grow=south west] { node[dbtree black node] {}}
    child[grow=right, level distance=12mm] { node[dbtree black node]{}
    child[grow=north east, level distance=6mm] { node[dbtree black node]{} }
    child [grow=south east, level distance=6mm] { node[dbtree black node]{} }
    };
    \draw[very thick, dotted] (-5mm,2.5mm) arc (150:210:5mm);
    \draw[very thick, dotted] (17mm, -2.5mm) arc (-30:30:5mm);
  \end{tikzpicture}
  \end{center}
\caption{The double bush free tree $t_{p,q}$ having $p$ leaves on one side and $q$ on the other}
\end{figure}

\noindent For $q=m-1$ the maximal number of branches from a vertex is $m$. We state the resulting conditions for energy preservation in the following lemma.
\begin{lemma}\label{lemma:DBmono}
Let $(A,b,c)$ be a Runge-Kutta scheme whose abscissae and weights satisfy the quadrature conditions \eqref{eq:quadRK} for $1\leq k\leq m$.
Then the conditions for energy preservation imposed on the method by the double bush free trees $t_{p,q}$, henceforth called the {\em double bush conditions} are
\begin{equation} \label{eq:DBmono}
pb^T C^{p-1}Ac^q - qb^T C^{q-1} A c^p=\frac{1}{q+1}-\frac{1}{p+1},\qquad 1\leq p<q\leq m-1.
\end{equation}
Here $c=(c_1,\ldots,c_s)^T$, $C=diag(c_1,\ldots,c_s)$. Powers of $c$ are defined componentwise.
A particular solution to the double bush conditions is given as
\begin{equation}\label{eq:avfmat}
     A_{\mathrm{avf}} = cb^T
\end{equation}
\end{lemma}
\begin{proof} The free tree $t_{p,q}$ is the equivalence class containing the four distinct rooted trees: $t_1=[[\ab^{p-1}[\ab^q]]]$, $t_2=[\ab^p[\ab^q]]$,
$t_3=[\ab^q[\ab^p]]$ and $t_4=[[\ab^{q-1}[\ab^p]]]$. 
 \begin{figure}[ht]
 \begin{center}
   \begin{tikzpicture}[dbtree]
    \node[dbtree black node] {}
    child[grow=90] {node[dbtree black node] {}
    child[grow=135] { node[dbtree black node] {} }
    child[grow=45] { node[dbtree black node] {} 
    child[grow=45, level distance=6mm] { node[dbtree black node]{}}
    child[grow=90, level distance=6mm] { node[dbtree black node]{}}
    child[grow=135, level distance=6mm] { node[dbtree black node]{}}}
    };
   \end{tikzpicture}
   \qquad
\begin{tikzpicture}[dbtree]
    \node[dbtree black node] {}
    child[grow=100] { node[dbtree black node] {}} 
    child[grow=150] {node[dbtree black node] {} }
    child[grow=45] { node[dbtree black node] {} 
    child[grow=45, level distance=6mm] { node[dbtree black node]{}}
    child[grow=75, level distance=6mm] { node[dbtree black node]{}}
    child[grow=105, level distance=6mm] { node[dbtree black node]{}}
    };
   \end{tikzpicture}
   \qquad
   \begin{tikzpicture}[dbtree]
    \node[dbtree black node] {}
    child[grow=90] { node[dbtree black node] {}} 
    child[grow=120] {node[dbtree black node] {} }
    child[grow=150] { node[dbtree black node] {} }
    child[grow=45, level distance=6mm] { node[dbtree black node]{}
    child[grow=55, level distance=6mm] { node[dbtree black node]{}}
    child[grow=105, level distance=6mm] { node[dbtree black node]{}}
    };
   \end{tikzpicture}
   \qquad
     \begin{tikzpicture}[dbtree]
    \node[dbtree black node] {}
    child[grow=90] {node[dbtree black node] {}
    child[grow=160] { node[dbtree black node] {} }
    child[grow=110] { node[dbtree black node] {} }
    child[grow=45, level distance=6mm] { node[dbtree black node]{}
    child[grow=65, level distance=6mm] { node[dbtree black node]{}}
    child[grow=115, level distance=6mm] { node[dbtree black node]{}}}
    };
   \end{tikzpicture}
\end{center}
\caption{The four trees in the double bush equivalence class $t_{2,3}$}
\end{figure}
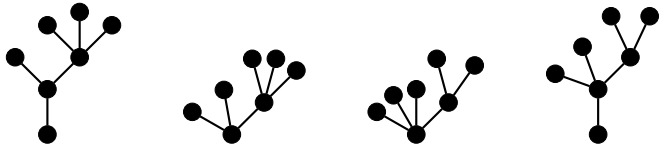

One has then
$\kappa(t_i,t_j)=|i-j|$. By the quadrature conditions,
\[
   a(B_-(t_2))=a(\ab)^pa([\ab^{q}])=\frac{1}{q+1},\quad
   a(B_-(t_3))=a(\ab)^qa([\ab^{p}])=\frac{1}{p+1}.   
\]
whereas
\[
  a(B_-(t_1))=a([\ab^{p-1},[\ab^q]]) = b^TC^{p-1}Ac^q,\qquad
  a(B_-(t_4))=a([\ab^{q-1},[\ab^p]]) = b^TC^{q-1}Ac^p.
\]
We also compute $\sigma(t_1)=(p-1)!q!$ $\sigma(t_2)=\sigma(t_3)=p!q!$,
and $\sigma(t_4)=p!(q-1)!$. Substituting all this into the conditions \eqref{eq:theo2} we get the stated conditions.

We finally substitute  \eqref{eq:avfmat} for $A$ in  \eqref{eq:DBmono} and use the quadrature conditions to get
\[
pb^TC^{p-1}cb^Tc^q - q b^TC^{q-1}cb^Tc^p =
\frac{p}{(p+1)(q+1)}-\frac{q}{(q+1)(p+1)}=\frac{1}{q+1}-\frac{1}{p+1}.
\] 
\end{proof}
\noindent For a given $s$, the double bush conditions \eqref{eq:DBmono} define a linear operator 
$M: \mathbf{R}^{s\times s} \rightarrow \mathbf{R}^{\frac12(m-1)(m-2)}$ acting on the set of Butcher matrices $A$. Generally, this operator depends on the quadrature coefficients$(b_i,c_i)$, as well as on $m$ and $s$. However, in our case we shall always be concerned with quadrature formulas of the highest possible order, so that we have either $m=2s$ or $m=2s-1$. We can then also represent the abscissae and weights of the quadrature formula by means of a single real parameter $\zeta$ as follows: We assume that $(c_1,\ldots,c_s)$ are the distinct zeros of the polynomial
$P_s(x)-\zeta P_{s-1}(x)$ where $P_q$ is the $q$th degree Legendre polynomial relative to
the interval $[0,1]$, and the weights are determined by solving \eqref{eq:quadcond} for $k\leq s$.
We would thus have $M=M(\zeta, s, m)$, but we shall restrict our attention the particular  even and odd cases for $m$: $M(0,s,2s)$ and $M(\zeta,s,2s-1)$ respectively. For ease of notation we still denote the linear operator simply by $M$ when it will be clear from the context whether we are considering the even or odd case. The following lemma is included without proof for future reference
\begin{lemma} \label{lemma:N1}
In both the even and odd cases, the matrix $N_1=(\mathbf{1}-c)b^T$ is in the kernel of $M$, i.e.
$M((\mathbf{1}-c)b^T) = 0$, where $\mathbf{1}=(1,\ldots,1)^T\in\mathbf{R}^s$.
\end{lemma}

It is useful to combine the double bush conditions into conditions involving arbitrary polynomials in $C$ and $c$ rather than the monomials used in the previous lemma, we shall let $\Pi_p$ denote the linear space of polynomials of degree at most $p$.
\begin{lemma} \label{lemma:DBpoly}
Let $(A,b,c)$ be a Runge-Kutta scheme whose abscissae and weights satisfy the quadrature conditions \eqref{eq:quadRK} for $1\leq k\leq m$. Assume that it also satisfies the double bush conditions \eqref{eq:DBmono}. Let $P\in\Pi_p$ and $Q\in \Pi_q$ such that $P(0)=Q(0)=0$. Then  
\begin{equation} \label{eq:DBpoly}
   b^TP'(C)AQ(c) - b^T Q'(C)AP(c) = P(1)\int_0^1 Q(t)\,\mathrm{d}t
   -Q(1)\int_0^1 P(t)\,\mathrm{d}t.
\end{equation}
\end{lemma}
 \begin{proof}
 Write $P(z)=\sum_{p'}\alpha_{p'}z^{p'}$ and $Q(z)=\sum_{q'}\beta_{q'}z^{q'}$.
 Then, using Lemma~\ref{lemma:DBmono}
 \begin{align*}
 b^T P'(C) A Q(c) - b^T Q'(C) A P(c) =
 \sum_{p',q'} \alpha_{p'}\beta_{q'} \left(
 p' b^T C^{p'-1} A c^{q'} - q' b^T C^{q'-1} A c^{p'}
 \right) \\
 =  \sum_{p',q'} \alpha_{p'}\beta_{q'} \left(
 \frac{1}{q'+1} - \frac{1}{p'+1} \right)
 =  P(1)\int_0^1 Q(t)\,\mathrm{d}t
   -Q(1)\int_0^1 P(t)\,\mathrm{d}t.
 \end{align*}
 \end{proof}
 \subsection{Some nonlinear conditions} \label{subsec:nonlincond}
 We shall introduce some conditions for energy preservation which are nonlinear in the Butcher matrix $A$, but that will be used in the final stage of the proof to eliminate the presence of elements from the kernel of $M$ in $A$ for an energy preserving integrator. 
 \begin{figure}[ht]
 \begin{center}
   \begin{tikzpicture}[dbtree]
    \node[dbtree black node] {}
    child[grow=up] { node[dbtree black node] {} }
    child[grow=north west] { node[dbtree black node] {} }
    child[grow=right, level distance=18mm] { node[dbtree black node]{}
    child[grow=north east, level distance=6mm] { node[dbtree black node]{} }
     child[grow=north, level distance=6mm] { node[dbtree black node]{} }
    child [grow=north west, level distance=6mm] { node[dbtree black node]{} }
    child [grow=right, level distance=18mm]{node[dbtree black node]{}
    child[grow=up, level distance=6mm]{node[dbtree black node]{} }
     child[grow=north east, level distance=6mm]{node[dbtree black node]{} }
    } 
    };
  \end{tikzpicture}
  \end{center}
  \caption{The triple bush free tree $t_{p,r,q}$ having $p\geq 1$ leaves on the left side, $q\geq 1$ leaves on the right side, and $r\geq 0$ leaves in the middle\label{fig:triplebush}}
\end{figure}
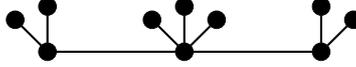
 The {\em triple bush trees}  yield conditions for energy preservation which are quadratic in $A$, an example of such a tree is shown in Figure~\ref{fig:triplebush}. Applying Theorem~\ref{theo:2}, and inserting the appropiate 
   B-series coefficients for Runge--Kutta methods \cite[section III.1.1]{hairer06gni},
 we find the {\em triple bush conditions}
 \begin{align}
 0&=pb^TC^{p-1}AC^rAc^q - b^TC^rAc^q+\frac{1}{(p+1)(q+1)} - rb^TC^{r-1}(Ac^q\circ Ac^p)
 \nonumber\\& - b^T C^r Ac^p + q b^TC^{q-1}AC^rAc^q 
 \end{align}
 Using the same approach as in Lemma~\ref{lemma:DBpoly} we derive the following 
 alternative version
 \begin{align}\label{eq:3bush}
0&= b^T P'(C)AR(C)AQ(c) + b^T Q'(C)AR(C)AP(c)
 -P(1)b^TR(C)AQ(c)\\&-Q(1)b^T R(C)AP(c)
 -b^TR'(C)(AQ(c) \odot AP(c)) + R(1)\int_0^1 P(t)\mathrm{dt}\int_0^1 Q(t)\mathrm{dt}
 \nonumber
 \end{align}
 where $P\in\Pi_p, Q\in \Pi_q, R\in\Pi_r$, $P(0)=Q(0)=0$,
  and where $p\leq m-1, q\leq m-1, r\leq m-2$. The symbol $\odot$ signifies component wise product between two vectors.
 
 Finally, we include a free tree and its corresponding condition used to investigate an exceptional case in Section~\ref{sec:oddcase}, see Figure~\ref{fig:asymbush}.
  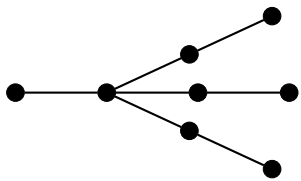
\begin{figure}[ht]
 \begin{center}
   \begin{tikzpicture}[dbtree]
    \node[dbtree black node] {}
     child[grow=right, level distance=12mm] { node[dbtree black node]{}
        child[grow=25, level distance=12mm]{ node[dbtree black node]{}
            child[grow=25, level distance=12mm]{ node[dbtree black node]{} } }
        child[grow=east, level distance=12mm] { node[dbtree black node]{} 
            child[grow=east, level distance=12mm]{ node[dbtree black node]{} }}
        child[grow=-25, level distance=12mm]{ node[dbtree black node]{}
             child[grow=-25, level distance=12mm]{ node[dbtree black node]{}}}
    };
  \end{tikzpicture}
  \end{center}
  \caption{ This tree has one single leaf to the left and $q$ branches of length two to the right}
 \label{fig:asymbush}
\end{figure}
The energy preservation condition corresponding to this tree is found to be
\begin{equation} \label{eq:asymbush}
b^T(Ac)^{q} - qb^TAC(Ac)^{q-1} + qb^TC(Ac)^{q-1}-
\left(\tfrac{1}{2}\right)^q=0,\quad 1\leq q\leq m-1.
\end{equation}

\subsection{\texorpdfstring{The case  $s=2$, $m=3$}{s=2,m=3}}
\label{subsec:s2m3}
 The general ideas of the proof can be illustrated for the case where a two stage Runge.-Kutta
 method is applied to problems with a cubic Hamiltonian. Then the
 abscissae, $c_1, c_2$ must be those of a third order quadrature rule. Thus, they are the zeros of $P_2(x)-\zeta P_1(x)$ for some $\zeta\in\mathbf{R}$, or equivalently the abscissae satisfy the condition   $3(c_1+c_2)-6c_1c_2=2$.  We represent any $2\times 2$ matrix in in the form
 $$
      A= \alpha_{1,1}\,\mathbf{1}b^T + \alpha_{2,1}\,cb^T+\alpha_{1,2}\,\mathbf{1}b^TC
      +\alpha_{2,2}\, cb^TC
 $$
  There is just one double bush condition \eqref{eq:DBmono} with $p=1, q=2$,
  $$
     b^TAc^2 - 2b^TCAc = 0.
  $$
 Substituting our form of $A$, and using the quadrature conditions with $k\leq 3$, we get
  $$
  \alpha_{11}
 + \alpha_{2,1}
 + (\tfrac1{2}-\tfrac{1}{6}\zeta)\,\alpha_{1,2}
+ ( \tfrac{7}{12}-\tfrac1{12}\zeta)\,\alpha_{2,2} = 0
  $$
 where we have made use of the fact that $b^Tc^3=\frac{1}{4}+\frac{1}{36}\zeta$. 
 The kernel is three dimensional,  a basis is given as
 $$
N_1= (\mathbf{1}-c)b^T,\quad
N_2= 2\mathbf{1}b^T + 3(\mathbf{1}-2c)b^TC,\quad
 N_3=2\zeta\mathbf{1}b^T + 3(7\mathbf{1}-6c)b^TC
 $$
 So any Butcher matrix candidate must be of the form
 $A=cb^T + \beta N$ where $N=v_1N_1+v_2N_2+v_3N_3$ and the row sum condition
 \eqref{eq:rowsum} then implies $N\mathbf{1}=0$.
  We obtain after some calculations that $N\in\ker M$ satisfying this condition must be a multiple of
  \begin{equation}
        N =  \left((\zeta-1)\mathbf{1}-2\zeta c\right)b^T(I-2C)
 \end{equation}
There are several possible nonlinear conditions to choose from in order to prove that any candidate solution of the form $A=cb^T + \beta N$ would require $\beta=0$. 
By taking $P(x)=G(x)=x(x-1)$ and $R(x)=1$ in \eqref{eq:3bush} and inserting our candidate solution, we find that $\frac{1}{81}\beta^2\zeta^3 =0$ which shows that one must have $\beta=0$
and $A=cb^T$ unless $\zeta=0$.

The remaining case $\zeta=0$ can be resolved by using the condition \eqref{eq:asymbush}
where upon inserting the expression $A=cb^T+\beta N$ one obtains the condition
$-\frac{1}{36}\beta^2(1+\zeta)^2=0$, therefore $\beta=0$ also for $\zeta=0$.

  
 \section{The case of even degree Hamiltonians} \label{sec:proof}
 In this section, we prove Theorem~\ref{theo:main} for the case that the polynomial Hamiltonian is of even degree $m=2s$, such that the underlying quadrature is the Gauss-Legendre formula. We use the following notation for the standard $L^2$ inner product between functions $u$ and $v$
 \[
      \langle u, v\rangle = \int_0^1 u(x)v(x)\,\mathrm{d}x
 \] 
For every non-negative integer $q$, we let $P_{q}$ be the Legendre polynomial of degree $q$ 
 \begin{equation}
    P_q(x) = \frac{1}{q!}\frac{\mathrm{d}^q}{\mathrm{d}x^q} x^q(x-1)^q,
 \end{equation}
 relative to the interval $[0,1]$, scaled such that $P_q(1)=1$ for every $q$, and consequently
 \begin{equation} \label{eq:legortho}
 \langle P_{k}, P_{\ell}\rangle = 
 \frac{\delta_{k\ell}}{2k+1}
 \end{equation} 
The polynomials 
 \begin{equation} \label{eq:Gpol}
     G_q(x) = \int_0^x P_{q-1}(t)\,\mathrm{d}t,\quad q\geq 1,
 \end{equation}
have for $q\geq 2$ the abscissae of the Gauss-Lobatto quadrature as zeros, and
  $$G_q(x)=\frac{1}{2(2q-1)}(P_q(x)-P_{q-2}(x))=\frac{1}{q(q-1)}x(x-1)P_{q-1}'(x).$$  
The following biorthogonality relations will be useful
\begin{equation}\label{eq:biortho}
\langle G_1,P_{\ell}'\rangle=1\;\forall \ell\in\mathbb{N},\qquad
\langle G_{q+1}, P_{\ell}' \rangle =-\frac{\delta_{\ell q}}{2q+1}\;\forall q\geq 2,
\ell\in\mathbb{N}.
\end{equation}
For any quadrature formula, we define the discrete counterpart to the inner product above
 \begin{equation} \label{eq:discip}
     \langle u, v\rangle_D = \sum_{i=1}^s b_i u(c_i) v(c_i) 
 \end{equation}
 and by a slight abuse of language we shall call it  the discrete inner product.
 If the quadrature formulas has order $m$ and $P$ and $Q$ are polynomials such that $\deg P + \deg Q \leq m-1$ then
  \begin{equation} \label{eq:disceqcont}
      \langle P,Q\rangle_D = \langle P, Q\rangle
 \end{equation}
 The discrete inner product can still be computed even in cases where it differs from
 the continuous one, the following result which will be of subsequent use, facilitates this in the case where $\deg P + \deg Q = m$.
 \begin{lemma} \label{lemma:discip}
 Suppose that a quadrature rule with abscissae $(c_1,\ldots,c_s)$  is exact for all polynomials of degree at most $m-1$, where $s\leq m\leq 2s$, and let $\rho_s=\prod_{i=1}^s (x-c_i)$. 
  Let $\pi_m$ be a monic polynomial of degree $m$.    
  Then
  \begin{equation}
  \langle \pi_m,1\rangle_D = \langle \pi_m,1\rangle - \langle \rho_s,\theta_{m-s}\rangle
  \end{equation}
for any monic polynomial $\theta_{m-s}$ of degree $m-s$.
 \end{lemma}
\begin{proof}
Let  $\delta_{m-1}=\pi_m-\rho_s\theta_{m-s}\in\Pi_{m-1}$ for an arbitrary monic polynomial 
$\theta_{m-s}\in\Pi_{m-s}$.  Then
\begin{equation}
\langle \pi_m,1\rangle_D = \langle \delta_{m-1} + \rho_s\theta_{m-s},1\rangle_D
=\langle \delta_{m-1},1\rangle_D=\langle\delta_{m-1},1\rangle
=\langle\pi_m,1\rangle - \langle \rho_s,\theta_{m-s}\rangle.
\end{equation}
 \end{proof}
 We can apply this lemma to obtain the following discrete inner products when Gauss-Legendre quadrature is used
  \begin{align} \label{eq:Dortho2s}
     \langle P_{2s-r}, P_r\rangle_D &= -\frac{\gamma_{2s-r}\gamma_r}{\gamma_s^2(2s+1)},
     \qquad
     \gamma_\ell=\frac{(2\ell)!}{\ell!^2}\\[2mm]
     \langle G_{2s-r+1},P_r'\rangle_D&=\frac{r}{2s-r+1}\langle P_{2s-r}, P_r\rangle_D.
       \label{eq:Dbiortho2s}
 \end{align}
 Here $\gamma_\ell$ is the leading coefficient of $P_\ell$.

 In analyzing the rank of the linear operator $M$, it is useful to work with the transformed double bush conditions given in Lemma~\ref{lemma:DBpoly}, equation \eqref{eq:DBpoly}. Generally, one may select any suitable set of polynomials so that the rank of $M$ is not reduced. In this section we shall make the choices $G_p$ and $G_q$ for $P$ and $Q$, where
 $1\leq p<q\leq m-1$. It will also be convenient to write the elements $A$ in terms of a basis
 as follows
 \begin{equation} \label{eq:Aform}
       A = \sum_{k=1}^s\sum_{\ell=1}^s \alpha_{k,\ell} A_{k,\ell},\qquad A_{k,\ell}= P_{k-1}(c)b^T P_{\ell}'(C)
 \end{equation}
 The resulting equations for the coefficients $\alpha_{k,\ell}$ when considering $M(A)=0$ are
 \begin{equation} \label{eq:Mcond}
 \sum_{k=1}^s\sum_{\ell=1}^s \alpha_{k,\ell}
 \left(
   \langle P_{p-1}, P_{k-1}\rangle_D
     \langle G_{q}, P_{\ell}'\rangle_D-
       \langle P_{q-1}, P_{k-1}\rangle_D
         \langle G_{p}, P_{\ell}'\rangle_D
          \right)    
         =0,\ 1\leq p< q\leq m-1.    
 \end{equation}
We prove the following result.
\begin{lemma} \label{lemma:rank}
Suppose $m=2s$, $s\geq 2$ and $c, b$ are the abscissae and weights of the Gauss-Legendre quadrature.  Then $\rank(M) = s^2-1$ and  $\ker M=\mbox{span}\{N_1\}$.
\end{lemma}
\begin{proof}
By Lemma~\ref{lemma:N1}, clearly $\rank(M)\leq s^2 -1$.   Note that the 
kernel element $N_1$ can be written in the format \eqref{eq:Aform} as 
$N_1=(\mathbf{1}-c)b^T=\frac14(A_{1,1}-A_{1,2})$.
Therefore, it must be true that $\rank(M)\geq s^2-1$ if some subset of the conditions \eqref{eq:Mcond} together with $\alpha_{1,1}=0$ cause the remaining $\alpha_{k,\ell}$ to vanish.
It is enough to consider just 
$s^2 -1$ (linearly independent) conditions among the $(2s-1)(s-1)$. 
We select the conditions corresponding to 
$1\leq p < q\leq 2s-1$, and such that $p+q\leq 2s+1$. The case $s=4$ is reported in figures~\ref{fig:figure1} and~\ref{fig:figure2}. The  $(p,q)$-element of the matrix in Figure~\ref{fig:figure1} corresponds to the condition in \eqref{eq:Mcond}.
The numbers refer to the ordering in which the conditions are used in the proof. The ones marked
$(n_a,n_b)$ are used simultaneously. The corresponding ordering of the unknowns $\alpha_{k,l}$ is reported in Figure~\ref{fig:figure2}.
\begin{figure}[ht]
\begin{minipage}[b]{0.45\linewidth}
\centering
\begin{tabular}{ccc}
 & & $q$\\
 & & $\longrightarrow$\\
$p$ & $\downarrow$ & $\left[
\begin{array}{ccccccc}
0 &10  & 9_a &  7_a & 4  & 7_b  & 9_b  \\
   & 0 & 8_a & 6_a & 3 &6_b  & 8_b \\
   &    & 0 & 5_a& 2 &5_b & -  \\
   &    &    & 0   & 1 &  -  & -  \\
   &    &    &    & 0   &  -  & -  \\
   &    &    &    &      & 0  & -  \\
   &    &    &    &      &     & 0  
\end{array}
\right]
$\\
\end{tabular}
\caption{Ordering of the conditions $(p.q)$.}
\label{fig:figure1}
\end{minipage}
\hspace{0.5cm}
\begin{minipage}[b]{0.45\linewidth}
\centering
\begin{tabular}{ccc}
 &          & $l$\\
 &          & $\longrightarrow$\\
$k$       & $\downarrow$& $
\left[
                                                \begin{array}{cccc}
                                                 0 & 9_a & 7_a& 4  \\
                                               10 & 8_a & 6_a &3 \\
                                              9_b  & 8_b & 5_a & 2  \\
                                                  7_b & 6_b  & 5_b & 1  \\
                                              \end{array}
                                              \right].
$\\
\end{tabular}
\caption{Ordering of  $\alpha_{k,l}$}
\label{fig:figure2}
\end{minipage}
\end{figure}

We begin by applying condition $(p,s+1)$ to $A_{k\ell}$ for $1\leq p\leq s$. Since \eqref{eq:disceqcont} applies for all inner products, we get from \eqref{eq:legortho} and \eqref{eq:biortho} that
\[
\langle P_{p-1}, P_{k-1}\rangle_D\,\langle P_\ell', G_{s+1}\rangle_D-
\langle P_{s}, P_{k-1}\rangle_D\,\langle P_\ell', G_{p}\rangle_D
=-\frac{s}{(s+1)(2s+1)(2k+1)}\delta_{pk}\delta_{s\ell}
\]
Thus  
$\alpha_{ps}=0$ for $1\leq p\leq s$.
We shall proceed by induction. Suppose it is true that
\begin{equation} \label{eq:indhyp}
    \alpha_{k\ell}=0,\qquad k>i+2,\ \ell>i+1,
\end{equation}
which is established for $i=s-2$. We now prove, by using conditions $(p,i+2)$ and $(p,2s-i)$ together with \eqref{eq:indhyp} that $\alpha_{k,\ell}=0$ for $k>i+1,\ \ell>i$. The first of these conditions applied to $A_{k\ell}$ yields
\[
\langle P_{p-1}, P_{k-1}\rangle_D
\langle P_{\ell}', G_{i+2}\rangle_D
-\langle P_{i+1}, P_{k-1}\rangle_D
\langle P_{\ell}', G_{p}\rangle_D
\]
 \eqref{eq:disceqcont} applies for all inner products
 and we conclude, using \eqref{eq:legortho} and \eqref{eq:biortho}, that condition $(p,i+2)$ implies, after multiplying both sides by $(2p-1)(2i+3)$
\begin{equation}\label{eq:feqpgt1}
    -\alpha_{p,i+1}+\alpha_{i+2,p-1}=0,
    \quad p>1
\end{equation}
and for $p=1$, multiplying each side by $2i+3$
\begin{equation}\label{eq:feqpeq1}
    -\alpha_{1,i+1}-\sum_{\ell=1}^s \alpha_{i+2,\ell}=0,\qquad
    p=1.
\end{equation}
We next consider condition $(p,2s-i)$ applied to $A_{k\ell}$ to get
\[
\langle P_{p-1}, P_{k-1}\rangle_D
\langle P_{\ell}', G_{2s-i}\rangle_D
-\langle P_{2s-i-1}, P_{k-1}\rangle_D
\langle P_{\ell}', G_{p}\rangle_D
\]
We readily compute $\langle P_{p-1}, P_{k-1}\rangle_D=\frac{\delta_{kp}}{2p-1}$ and
$\langle P_{\ell}',G_p\rangle_D=\frac{\delta_{p,\ell+1}}{2p-1}$ if $p>1$, and
$\langle P_{\ell}',G_1\rangle_D=1$.
For $\langle P_\ell',G_{2s-i}\rangle_D$, \eqref{eq:disceqcont} applies when $\ell\leq i$ causing it to vanish by \eqref{eq:biortho}. For $\ell=i+1$ the total degree equals $2s$, and by \eqref{eq:Dbiortho2s}
\[
    \langle P_{i+1}',G_{2s-i}\rangle_D=
    \frac{i+1}{(2s-i)}\langle P_{2s-i-1}, P_{i+1}\rangle_D
\]
Nonzero entries for $\ell>i+1$ can be ignored due to the induction hypothesis. Finally,
for $k\leq i+1$, one has $\langle P_{2s-i-1}, P_{k-1}\rangle_D=\langle P_{2s-i-1}, P_{k-1}\rangle = 0$.
For $k=i+2$, we invoke \eqref{eq:Dortho2s} just to assert that 
$\langle P_{2s-i-1}, P_{i+1} \rangle_D \neq 0$ so that this factor can be cancelled in the
condition $(p,2s-i)$ and we get
\begin{equation} \label{eq:seqpgt1}
-\frac{i+1}{2s-i}\alpha_{p,i+1}+\alpha_{i+2,p-1}=0,\quad p>1,
\end{equation}
and
\begin{equation} \label{eq:seqpeq1}
-\frac{i+1}{2s-i}\alpha_{1,i+1}-\sum_{\ell=1}^s \alpha_{i+2,\ell}=0,\quad p=1.
\end{equation}
Combining \eqref{eq:feqpgt1} and \eqref{eq:seqpgt1} we get for $p>1$ the system
\[
\left[
\begin{array}{cc}
-1 & 1 \\ 
-\frac{i+1}{2s-i} & 1
\end{array}
\right]
\left[
\begin{array}{c}
\alpha_{p,i+1}\\ \alpha_{i+2,p-1}
\end{array}
\right]
=0.
\]
and thus
\[
    \alpha_{p,i+1}=\alpha_{i+2,p-1}=0,\ p=2,\ldots,i+1.
\]
The remaining indices to be dealt with in the induction step are $\alpha_{1,i+1}$ and
$\alpha_{i+2,i+1}$ and for these we consider \eqref{eq:feqpeq1} and \eqref{eq:seqpeq1}.
We use that the only element in the ($i+2$)th row ($\alpha_{i+2,\ell}$) not yet found to be zero is 
$\alpha_{i+2,i+1}$ and so we obtain also in the case $p=1$ a nonsingular $2\times 2$ system
where the two unknowns must satisfy $\alpha_{1,i+1}=\alpha_{i+2,i+1}=0$.
The induction step is completed. The induction proof ensures that all $\alpha_{k\ell}=0$ except 
possibly $\alpha_{1,1}$ and $\alpha_{2,1}$, but the former is zero by assumption. But the remaining unused condition $(p,q)=(1,2)$ applied to $A_{2,1}$ yields $-1/3$ and thus also $\alpha_{2,1}=0$.
In summary, we have proved that for an $s\times s$-matrix $A$ is expressed in the form
\eqref{eq:Aform} with $\alpha_{1,1}=0$, the conditions \eqref{eq:Mcond} imply $A=0$ which is equivalent to $\rank(M)\geq s^2-1$. Combined with the known null-vector $N_1=(\mathbf{1}-c)b^T$ this proves that the rank of $A$ is precisely $s^2-1$.
 
\end{proof}  
  
\begin{proof} Theorem~\ref{theo:main} (even case).  
From Lemma~\ref{lemma:DBmono} \eqref{eq:avfmat} and Lemma~\ref{lemma:rank} we know that any solution to the double bush conditions for $m=2s\geq 3$ must be of the form $A=cb^T + \beta (\mathbf{1}-c)b^T$. But then the condition \eqref{eq:rowsum} immediately implies that
$$
      A\, \mathbf{1} = cb^T\mathbf{1} +\beta (1-c)b^T\mathbf{1}\quad\Rightarrow\quad
     \beta (\mathbf{1}-c)=0
$$
so that $\beta=0$ and we are left with the AVF method.
\end{proof}
  
\section{The case with odd degree} \label{sec:oddcase}
Suppose now that the degree $m$ of the Hamiltonian is odd. One still needs $c$ and $b$ which satisfy the quadrature conditions to order $m-1$. This means that it is necessary for the Runge-Kutta method to have at least $s=(m+1)/2$ stages such that $m\leq 2s-1$. On the other hand, choosing
$c$ and $b$ to be such quadrature points and letting $A=cb^T$ we have an energy preserving scheme with the minimal number of stages. We need to answer whether it is unique.

The strategy will be the same as in the even case. Now we assume that the quadrature rule consists of abscissae $(c_1,\ldots,c_s)$ which are the zeros of the polynomial $P_s-\zeta P_{s-1}$ for some real $\zeta$ and that $b_1,\ldots,b_s$ satisfy the quadrature conditions \eqref{eq:quadcond} for $k\leq s$.
The cases $\zeta=-1$ and $\zeta=1$ correspond to the Radau I and II formulas in which one has
$c_1=0$ and $c_s=1$ respectively. Of course $\zeta=0$ yields the Gauss-Legendre formula.
The case $s=2, m=3$ was discussed in Subsection~\ref{subsec:s2m3}, and in this section we sometimes assume tacitly that $s\geq 3$.
 
We let $R_l(x)$ $l=1,2,\dots$ be the polynomials of degree $l$ defined by
$$R_l(x)=\left\{\begin{array}{ll}
P_l(x), & l=0,\dots, s-1,\\
P_s(x)-\zeta P_{s-1}(x),& l=s,\\
R_s(x)P_{l-s}(x),&l\ge s+1.\\
\end{array}\right. 
$$
We consider also the polynomials $F_q$, defined for every positive integer as
 \begin{equation} \label{eq:Fpol}
     F_q(x) = \int_0^x R_{q-1}(t)\,\mathrm{d}t,
 \end{equation}
 so $F_q=G_q$ for $q=1,\dots ,s.$
 We observe that for $r\ge s$ we have
\begin{equation}
\label{PRr}
\langle P, R_r \rangle_D=0,
\end{equation}
for any polynomial $P$ of any degree. 

We will use the following explicit expressions for discrete inner products of polynomials.
\begin{lemma}\label{innerproductsF} For any real $\zeta$ we have
\begin{equation}
\label{ipsrsr}
\langle P_{s-r}',F_{s+r}\rangle_D= \zeta \,\frac{2s}{(2s-1)(s+r)}\, \frac{\gamma_{r-1}\gamma_{s-r}}{\gamma_{s-1}}\neq 0 \,\mathrm{when}\, \zeta\neq 0,
\end{equation}
\begin{equation} \label{eq:ipFsprPDsprm1}
\langle F_{s+r}, P_{s+1-r}'\rangle_D =
-\frac{\gamma_{r-1}\gamma_{s+1-r}}{\gamma_s (s+r)}\left(1+\zeta^2\frac{s+1-r}{(s+r-1)(2s-1)}\right)
\end{equation}
  \begin{equation}
  \label{}
  \langle P_s',F_{s+2}\rangle_D=-\,\zeta\,\frac{s}{(2s-1)(s+1)}\,\left(    \frac{s}{(2s-1)(s+2)}\zeta^2-1 \right),
  \end{equation}

\end{lemma}
\begin{proof}
The proof relies on Lemma  \ref{lemma:discip} and the three term recursion formulae for Legendre polynomials.
\end{proof}

\begin{lemma} \label{lemma:Rsp}
For all $\zeta\neq -1$ 
there exists a polynomial $\tilde{P}_s$ of degree less than or equal to $s$ such that
\begin{equation}
\label{orthogonality:Rsp}
\langle \tilde{P}_s',F_{s+r}\rangle_D=0,\,\, r=1,\dots , s-2,\quad
\langle \tilde{P}_s',G_{1}\rangle=0, 
\end{equation}
and
\begin{equation}
\label{eq:G2nonzero}
\langle \tilde{P}_s',G_{2}\rangle\neq 0.
\end{equation}

\end{lemma}
\begin{proof}
Consider the  vectors 
$f_{s+r}:=F_{s+r} ( c )\in\mathbf{R}^s$, $r=1,\dots , s-2$ and $g_1:=G_1( c )$. Then 
$$\mathrm{dim}\left(\mathrm{span}\{ f_{s+1}, \dots , f_{2s-2}, g_1  \}\right)=\gamma \le s-1,\quad \mathrm{dim}\left(\mathrm{span}\{ f_{s+1}, \dots , f_{2s-2}, g_1  \}^{\perp_D}\right)=s-\gamma \ge 1.$$
The superscript ``$\perp_D$''  denotes the complementary vector space with respect to $\langle \cdot , \cdot \rangle_D$ interpreted as an inner product on $\mathbf{R}^s$. Then there exists $\tilde{g}\neq 0$ and
$\tilde{g}\in \mathrm{span}\{ f_{s+1}, \dots , f_{2s-2}, g_1  \}^{\perp_D}$.
We  write $\tilde{g}$ by using the basis $P'_1( c ),\dots , P'_s( c )$ of $\mathbf{R}^s$, i.e.
$$\tilde{g}=\sum_{\ell=1}^sv_\ell P_\ell'( c ).$$
We define $\tilde{P}_s:=\sum_{\ell=1}^s v_\ell P_\ell'$. By construction this polynomial satisfies the orthogonality conditions of (\ref{orthogonality:Rsp}).
The condition (\ref{eq:G2nonzero}) is
$$\langle G_2,\tilde{P}_s\rangle_D=\sum_\ell v_\ell\,\langle G_2,P_\ell'\rangle_D=-\frac{1}{3}v_1,$$
and is nonzero if and only if $v_1\neq 0$.
We now prove that  if $\zeta\neq -1$  then $v_1\neq 0$.
Consider the $(s-1)\times s$ matrix $\Gamma$ with entries
$$\Gamma_{i,j}:=\langle F_{s+i},P_j'\rangle_D,\, i=1,\dots , s-2, \,j=1,\dots , s,\quad \Gamma_{s-1,j}:=\langle G_1,P_j'\rangle_D=1,\, j=1,\dots , s.$$
We observe that
$$\Gamma_{i,j}=0,\quad i+j\le s-1.$$
The conditions \eqref{orthogonality:Rsp} for
 the vector $v:=[v_1,\dots,v_s]^T$ can be written as
$$\Gamma\, v=0.$$
Let us define $\bar{\Gamma}$ to be the $(s-1)\times (s-1)$ Hessenberg matrix whose columns are the last $s-1$ columns of $\Gamma$, and denote by $\bar{v}$ the vector $\bar{v}:=[v_2,\dots,v_s]^T$, then we have
$$\bar{\Gamma}\,\bar{v}=-v_1\, \Gamma\, e_1,$$
where $e_1$ is the first canonical vector in $\mathbf{R}^s$. We also note that $\Gamma\, e_1=e_1$. If $\zeta =0$, $\bar{\Gamma}$ is upper triangular because the entries $\Gamma_{i,s-i}=\langle F_{s+i},P'_{s-i}\rangle_D=0$,  $i=1,\dots , s-2$, and $\bar{\Gamma}$ is invertible because $\langle F_{s+i},P'_{s-i+1}\rangle_D\neq 0$ for $i=1,\dots , s-2$, see Lemma \ref{innerproductsF}.

If  $\zeta\neq 0$, due to the Hessenberg form of $\bar{\Gamma}$ and the fact that $\Gamma_{i,s-i}=\langle F_{s+i},P'_{s-i}\rangle_D\neq 0$, 
one concludes that  $\bar{\Gamma}$ is invertible if and only if the two last columns of $\bar{\Gamma}$ are linearly independent.
By Lemma \ref{innerproductsF},   if $\zeta \neq -1$,  the two determinants
$$\mathrm{det}\left[\begin{array}{cc}
\langle F_{s+1} ,P'_{s-1}\rangle_D & \langle F_{s+1} ,P'_{s}\rangle_D\\
1 & 1\\
\end{array}\right], \quad \mathrm{det}\left[\begin{array}{cc}
\langle F_{s+2} ,P'_{s-1}\rangle_D & \langle F_{s+2} ,P'_{s}\rangle_D\\
1 & 1\\
\end{array}\right]\,$$
cannot be simultaneously zero.  Thus,  $\mathrm{det}\,\bar{\Gamma}\neq 0$, and we can write 
$$\bar{v}=-v_1\bar{\Gamma}^{-1}\, e_1.$$
As a consequence $v=[v_1,\bar{v} ]^T$ is a non trivial solution of $\Gamma v=0$ giving $\tilde{g}=\sum_{\ell=1}^sv_\ell P_\ell'( c ) \neq 0$, if and only if $v_1\neq 0$.
\end{proof}
 In the sequel we will use $\tilde{P}_s$ as characterized in Lemma~\ref{lemma:Rsp} 
 whenever $\zeta \neq -1$.
When $\zeta =1$, one may take $\tilde{P}_s=P_s+P_{s-1}-2P_1$,  and when $\zeta =0$, $\tilde{P}_s=P_2-P_1$. When $\zeta=-1$, we will use instead $\tilde{P}_s:=P_s-P_{s-1}$, for which only (\ref{orthogonality:Rsp}) hold, but  not (\ref{eq:G2nonzero}) as $\langle \tilde{P}_s',G_{2}\rangle=0$.

We express any $A$ in a form  similar to (\ref{eq:Aform}),
\begin{equation}
\label{eq:Abasisnew}
A=\sum_{k,\ell\le s}\alpha_{k,\ell}P_{k-1}( c )b^T\tilde{P}_\ell' ( C ),
\end{equation}
where we use the notation $\tilde{P}_\ell:=P_\ell$ for $l\le s-1$,  if $\zeta\neq 0$, and  $\tilde{P}_\ell:=P_\ell$ for $l\neq 2$ and $\tilde{P}_2:=P_s$,  if $\zeta =0$.

\begin{lemma} \label{lemma:rank10}
Suppose $m=2s-1$, $s\geq 3$ and $c, b$ are the abscissae and weights of the Radau quadrature with  $\zeta\in\mathbf{R} \setminus \{ -1\}$.  Then $\rank(M) \geq s^2-3$. If  $\zeta=-1$, then $\rank(M) \geq s^2-s-1.$
\end{lemma}
\begin{proof}
For each $1\le p<q\le 2s-2$ 
we consider the condition obtained from Lemma~\ref{lemma:DBpoly}
by choosing $P(x)=F_p(x)$ and $Q(x)=F_q(x)$ as defined in \eqref{eq:Fpol}. 
Condition $(p,q)$ applied to $A$ is
\begin{equation}
\label{eq:transDBpq}
m_{p,q}(A)=\sum_{k,\ell}\alpha_{k,\ell}\left( \langle R_{p-1}, P_{k-1} \rangle_D\, \langle \tilde{P}_\ell',F_q\rangle_D-\langle R_{q-1}, P_{k-1} \rangle_D\, \langle \tilde{P}_\ell',F_p\rangle_D \right). 
\end{equation}
Due to the orthogonality properties of  the polynomials $R_p$ and $F_q$, 
all conditions vanish identically for  $ s+1\le p <q.$
We show that if $\alpha_{1,s}=\alpha_{2,1}=\alpha_{2,s}=0$ and $\zeta\in\mathbf{R} \setminus \{ 0,-1\}$, then $ m_{p,q}(A) =0$ for $1\leq p < q \le 2s-2$ implies that $\alpha_{k,\ell}=0$ for all $k,\ell=1,\dots ,s$. Analogously, if $\zeta=-1$, $\alpha_{1,s}=\alpha_{2,1}=\alpha_{2,s}=0$ and in addition $\alpha_{\ell,s}=0$ for $\ell=3,\dots ,s$, then $ m_{p,q}(A) =0$ for $1\leq p < q \le 2s-2$ implies $\alpha_{k,\ell}=0$ for all $k,\ell=1,\dots ,s$.

It is enough to consider a subset of  linearly independent conditions: $s^2-3$, $\forall \zeta\in\mathbf{R} \setminus \{ 0,-1\}$; and $s^2-s-1$ for $\zeta=-1$ respectively.  

These are all the conditions corresponding to $(p,q)=(j,s+r)$, $r=1,\dots,s-2$ and $j=1,\dots, s$ and successively $(p,q)=(j,s-r)$, $r=0,\dots , s-j-1$, and $j=1,2$. In figures \ref{fig:figure3} and \ref{fig:figure4} we show the ordering of the conditions and of the unknowns in the case $s=5$.
The numbers refer to the ordering in which the conditions are used in the proof, and in which the $\alpha_{k,\ell}$ will be shown to vanish.
\begin{figure}[ht]
\begin{minipage}[b]{0.45\linewidth}
\centering
\begin{tabular}{ccc}
 & & $q$\\
 & & $\longrightarrow$\\
$p$ & $\downarrow$ & $\left[\begin{array}{cccccccc}
                0 &19  & 18 &  17 & 16  & 1  & 2  & 3  \\
                & 0 & 22 &21 & 20 & 4 & 5 &6 \\
   &    & 0 & -& - & 7  & 8 &9  \\
   &    &    & 0   & - &  10 &   11&12 \\
   &    &    &    & 0   & 13 &  14& 15 \\
   &    &    &    &      & 0  & -  &- \\
      &    &    &    &      &   & 0  &-  \\
         &    &    &    &      &   &  & 0\\
\end{array}\right]
$
\\
\end{tabular}
\caption{Ordering of the conditions $(p.q)$.}
\label{fig:figure3}
\end{minipage}
\hspace{0.5cm}
\begin{minipage}[b]{0.45\linewidth}
\centering
\begin{tabular}{ccc}
 &          & $l$\\
 &          & $\longrightarrow$\\
$k$       & $\downarrow$& $
\left[
                                                \begin{array}{ccccc}
                                               19 & 3 & 2 & 1  & 0 \\
                                               0  & 6 & 5 &4 & 0\\
                                              18  & 9 & 8 & 7 & 22 \\
                                              17 & 12 & 11 & 10 & 21 \\
                                              16 & 15 & 14 & 13 & 20 \\
                                              \end{array}
                                              \right].
$\\
\end{tabular}
\caption{Ordering of  $\alpha_{k,l}$}
\label{fig:figure4}
\end{minipage}
\end{figure}

We start considering the case $\zeta \neq 0$, the proof is similar in the case $\zeta =0$ and we will highlight later the difference.  We first consider conditions $(p,q)=(j,s+r)$ and the unknowns $\alpha_{j,s-r}$, for a fixed $j$ and $r=1,\dots ,s-2$, and $j=1,\dots ,s$. 
For these conditions, due to  (\ref{PRr}), (\ref{eq:transDBpq}) simplifies into
\begin{equation}
\label{eq:conditionjsr}
\sum_{l}\alpha_{j,\ell}\,\langle \tilde{P}'_\ell,F_{s+r}\rangle_D =0,\qquad j=1,\dots, s.
\end{equation}
The last term of the sum at the left hand side vanishes due to $\langle \tilde{P}_s',F_{s+r}\rangle_D=0,$ see (\ref{orthogonality:Rsp}).
 For $\ell\le s-(r+1)$ we obtain that
 $$\langle P'_\ell,F_{s+r}\rangle_D =\langle P'_\ell,F_{s+r}\rangle=\int_0^1P'_\ell(t)F_{s+r}(t)\, dt=0,$$
 this follows using integration by parts and the orthogonality properties of the Legendre polynomials.
So for $r=1$ we get the equation
 $$\alpha_{j,s-1}\langle P'_{s-1},F_{s+1} \rangle_D=0,$$
 implying that $\alpha_{j,s-1}=0$ by (\ref{ipsrsr}).  Proceeding by induction over $r$  
 we similarly obtain that
 $$\alpha_{j,s-r}\langle P'_{s-r},F_{s+r} \rangle_D=0, \qquad r=2,\dots, s-2,$$
 and by (\ref{ipsrsr}), this in turn implies that $\alpha_{j,s-r}=0$  respectively for $r=2,\dots , s-2$.

For $\alpha_{j,1}$  and $j=3,\dots, s$, we consider the conditions $p=1$, $q=j$ leading to the equations
 $$\alpha_{j,1}\, \langle P'_1, G_1 \rangle=0,$$
 implying $\alpha_{j,1}=0$, since $\langle P'_1, G_1 \rangle=1$. Here we have used that $\alpha_{1,\ell}=0$ for $l=2,\dots, s$, $\langle P_1',G_j\rangle=0$, for $j=3,\dots, s$ and $\langle \tilde{P}_s',G_1\rangle=0$.

For $\alpha_{1,1}$ we use $(p,q)=(1,2)$, using $\alpha_{1,\ell}=0$ for $\ell=2,\dots ,s$ and $\alpha_{2,\ell}=0$ for $\ell=1,\dots ,s$ we simplify the equation into
$$\alpha_{1,1}\,\langle P_1',G_2\rangle =0,$$
giving $\alpha_{1,1}=0$ since $\langle P_1',G_2\rangle\neq 0$. The proof is here concluded for the case $\zeta= -1$.

We next consider conditions $(p,q)=(2,s-r)$, $r=0,\dots, s-3$, leading to the equations
$$\frac{1}{3}\sum_\ell\alpha_{2,\ell}\langle \tilde{P}_\ell',G_{s-r}\rangle_D-\frac{1}{2(s-r)-1}
\sum_\ell\alpha_{s-r,\ell}\langle\tilde{P}_\ell',G_2\rangle_D=0,$$
since we have shown that $\alpha_{2,\ell}=0$ for $\ell=2,\dots , s-1$ and $\alpha_{s-r,\ell}=0$, $\ell=1,\dots , s-1$, and by hypothesis, $\alpha_{2,1}=\alpha_{2,s}=0$.  We are left with the equation
$$\alpha_{s-r,s}\, \langle \tilde{P}_{s}',G_2\rangle=0.$$
By (\ref{orthogonality:Rsp}) $\langle \tilde{P}_{s}',G_2\rangle\neq 0$  for $\zeta\neq -1$,
so we conclude $\alpha_{s-r,s}=0$, $r=0,\dots , s-3$.

This concludes the proof, for the cases $\zeta\in\mathbf{R}\setminus \{ 0 \}$. In the case $\zeta =0$, $\tilde{P}_1'( c ),\dots , \tilde{P}_s'( c )$ is no longer a basis of $\mathbf{R}^s$ because  $\tilde{P}_s=P_2-P_1$ . We then consider $\tilde{P}_2:=P_s$ in (\ref{eq:Abasisnew}) and proceed as in the case, $\zeta\in \mathbf{R} \setminus \{ 0,-1 \}$. The equation (\ref{eq:conditionjsr}) becomes in this case
$$\alpha_{j,2}\langle \tilde{P}_2',F_{s+r}\rangle_D+\sum_{\ell=s-r+1}^{s-1}\alpha_{j,\ell}\langle P_\ell',F_{s+r}\rangle_D=0,$$
leading to $\alpha_{j,2}=0$ for $r=1$ and, by induction, to $\alpha_{j,s-r}=0$ for $1<r\leq s-2$, by using that $\langle \tilde{P}_{s-r+1}',F_{s+r}\rangle_D\neq 0$. The rest of the proof is the same as in the previous case.
\end{proof}

The following two discrete inner products are easy consequences of Lemma~\ref{lemma:discip}
\begin{align}
\langle P_{s+r-1}, P_{s-r}\rangle_D &= \frac{\gamma_{s+r-1}\gamma_{s-r}}{\gamma_s\gamma_{s-1}}
\frac{\zeta}{2s-1} \label{eq:discip1} \\
\langle G_{s+r}, P_{s-r}' \rangle_D &= \frac{s-r}{s+r} \langle P_{s+r-1}, P_{s-r}\rangle_D
\label{eq:discip2}
\end{align}
To establish that the bounds for the rank of $M$ are sharp, we shall simply derive a 
suitable set of linearly independent matrices in the 
kernel of $M$, and we make the ansatz that these kernel elements are all rank one matrices.
We use the generic representation
\begin{equation} \label{eq:UVform0}
    U(c)b^T V(C),\quad U(x)=\sum_{k=1}^s u_k, P_{k-1}(x),\quad
    V(x) = \sum_{\ell=1}^s v_\ell P_\ell'(x)
\end{equation}
and for convenience, we shall define $\displaystyle{v_0 := -\sum_{\ell=1}^s v_\ell}$.
\begin{lemma} \label{lemma:kerelts}
Let $s\geq 2$, $m=2s-1$ and suppose that the quadrature rule $(c,b)$ has order $2s-1$ where
$c_1,\ldots,c_s$ are the zeros of the polynomial $P_s(x)-\zeta P_{s-1}(x)$ for some real $\zeta$.
Then, if $\zeta\neq -1$ there exists a basis for $\ker M$  consisting of matrices
$$
    N_i = U^i(c) b^T V^i(C),\quad i=1,2,3.
$$
for polynomials $U^i\in\Pi_s$, $V^i\in\Pi_s$ that can be written in the form
\begin{equation} \label{eq:UVform}
      U^i(x) = \sum_{k=1}^s u_k^{(i)} P_{k-1}(x),\qquad
      V^i(x) = \sum_{\ell=1}^s v_\ell^{(i)} P_\ell'(x)
\end{equation}
having the properties\footnote{The first basis element $(i=1)$ is nothing else than $4N_1=4(\mathbf{1}-c)b^T$ so that 
all the coefficients $u_k^{(1)}, v_\ell^{(1)}$ not listed are zero.}
\begin{align*}
i=1: && u_1^{(1)}&=1, &u_2^{(1)}&=-1, & v_1^{(1)}&=1,\\
i=2: && u^{(2)}_1&=1, &u_2^{(2)}&=0, & v_1^{(2)}&=0,\\  
i=3:&& u_1^{(3)}&=0, &u_2^{(3)}&=1,   &v_0^{(3)}&:=\sum_{\ell}v_\ell^{(3)} = 0.\\
\end{align*}
If $\zeta\neq 0$, then one may choose $v_s^{(i)}=1,\ i=2,3$.
\end{lemma}
\begin{proof}
We first intend to show that the equations
\begin{equation}\label{eq:kercond}
   \sum_{k,\ell=1}^s u_kv_\ell\left(\langle P_{p-1},P_{k-1}\rangle_D
   \langle G_q,P_\ell'\rangle_D-
   \langle P_{q-1},P_{k-1}\rangle_D\langle G_p,P_\ell'\rangle_D \right)= 0,\quad
   1\leq p<q\leq 2s-2,
\end{equation}
have a solution with the given preset values from the lemma. 
We already know from Lemma~\ref{lemma:N1} that $U^1(c)b^TV^1(C)$ is a kernel element so henceforth we consider only $i=2,3$.

We treat the case $\zeta=0$ separately.  Consider
\begin{align}
u^{(2)} &= (1,0,-1,0,\ldots,0)^T,& u^{(3)} &= (0,1,-1,\ldots,0)^T,\label{eq:GLu2u3} \\
 v^{(2)}&=(0,1,0,\ldots,0)^T,  & v^{(3)}&=(1,-1,\ldots,0)^T. \label{eq:GLv2v3}
\end{align}
In these cases, all discrete inner products in \eqref{eq:kercond} equal the continuous ones for all $(p,q)$ because $\zeta=0$ implies that the quadrature rule has order $2s$, and all conditions are readily verified.

Until the end of this proof, we assume that $\zeta\neq 0$. We next use the conditions
$(p,q)$ of \eqref{eq:kercond}, setting $p=1,2$ and $3\leq q\leq s$.  We apply
\eqref{eq:legortho}, \eqref{eq:biortho} and for $q=s$ we use also \eqref{eq:discip2} with $r=0$.
We get
\begin{align}
v_{p-1} u_k &= u_p v_{k-1},\ 3\leq k\leq s-1,\label{eq:uk}\\
v_{p-1} u_s &= u_p(v_{s-1}-\zeta v_s), \label{eq:us}
\end{align}
The conditions with $p=1$ are useful for kernel elements of type $i=2$, and
those with $p=2$ can be used for the $i=3$ type.
We now select the conditions $(p,s+r)$ in
\eqref{eq:kercond} where $p=1,2$ and $1\leq r\leq s-2$. In this case, because of
\eqref{eq:legortho} and \eqref{eq:disceqcont}, the discrete inner products will vanish whenever
$\ell\leq s-r-1$ and $k-1\leq s-r-1$. We get
\begin{equation}
    u_p\sum_{\ell=s-r}^s v_\ell\langle G_{s+r},P_\ell' \rangle_D
    +  v_{p-1} \sum_{k=s-r+1}^s u_k\langle P_{s+r-1}, P_{k-1}\rangle_D=0,\ p=1,2.
\end{equation}
Substitute \eqref{eq:uk} and \eqref{eq:us} into these conditions and change summation index
 \begin{equation} \label{eq:condforv}
 u_p  \left(\sum_{\ell=s-r}^s v_\ell\, \langle G_{s+r}, P_{\ell}'\rangle_D 
   +  \sum_{\ell=s-r}^{s-1} v_{\ell}\, \langle P_{s+r-1}, P_{\ell}\rangle_D 
   - \zeta\, v_s\,  \langle P_{s+r-1}, P_{s-1}\rangle_D \right)=0, 
 \end{equation}
 for $\quad 1\leq r\leq s-2$ and $p=1,2$. We have thus derived $s-2$ conditions for the $s-1$ unknowns $v_2,\ldots,v_s$. Choosing $v_s=1$ there is an upper triangular system to be solved for 
 $v_2,\ldots, v_{s-1}$. This system is nonsingular, since the pivot elements, which can be computed explicitly from \eqref{eq:discip1} and \eqref{eq:discip2}, are nonzero whenever $\zeta\neq 0$.
 Note that  one must choose $p=1$ for the kernel element of type $i=2$ and $p=2$ for the $i=3$ type
 to have $u_p=1$. So we conclude that $v_\ell^{(2)}=v_\ell^{(3)}$ for 
 $2\leq\ell\leq s$. In order to solve for $u^{(i)}_k$, $i=2,3$ from \eqref{eq:uk}-\eqref{eq:us}, we need to make sure that $v_0^{(2)}=-\sum_\ell v_\ell^{(2)}\neq 0$ for $i=2$, and
 $v_1^{(3)}\neq 0$ for $i=3$. In the former case, note that $v_0^{(2)}=0$ would lead
 to $v_\ell^{(2)}=0,\ \ell=1,\ldots,s-2$ \eqref{eq:uk}, $v_{s-1}^{(2)}=\zeta$ \eqref{eq:us}, and thereby
 $-v_0^{(2)}=v_{s-1}^{(2)}+v_s^{(2)}=\zeta+1=0$ which is the exceptional case excluded in the lemma. Similarly, $v_1^{(3)}=0$ would also imply $\zeta=-1$. We conclude that all parameters 
 $u_k^{(i)}, v_\ell^{(i)}$ have been determined for $\zeta\neq 0$. It is left to the reader to verify that with these choices, all the remaining conditions $(p,q),\ 3\leq p<q\leq 2s-2$ are consequently satisfied.
 
 We already know from Lemma~\ref{lemma:rank10} that $\rank(M)\geq s^2-3$ for $\zeta\neq -1$, it therefore just remains to check that $N_i, i=1,2,3$ form a linearly independent set. Suppose
 $$
    \sum_{i=1}^3 \alpha_i U^i(c)b^T V^i(C) = 0.
 $$
 Multiplying this matrix from the right by the vector $c$, we obtain
 \[
      \alpha_1  U^1(c) - \alpha_2 v_0^{(2)} U^2(c) = 0,
 \]
 where we have used that $\langle V^3, G_1\rangle_D=-v_0^{(3)} = 0$.
 Thus, since $v_0^{(2)}\neq 0$ for $\zeta\neq -1$ and since clearly $U^1(c)$  and $U^2(c)$ are linearly  independent we conclude that $\alpha_1=\alpha_2=0$ so that also $\alpha_3=0$.
\end{proof}
As a by-product of the proof of the previous lemma, we obtain
\begin{lemma} \label{lemma:rankkerelts}
$\ker M$ consists of matrices $A$ such that $\rank A\leq 2$. More precisely, for $\zeta\neq -1$, one has
\begin{equation}
    U^{3}(x) = U^2(x)-U^1(x),\quad V^3(x)=V^2(x)+v_1^{(3)}V^1
\end{equation}
\end{lemma}
We now consider the kernel of $M$ for the exceptional case $\zeta=-1$.
\begin{lemma}
For $\zeta=-1$ $\rank(M)=s^2-s-1 $. The following $s+1$ matrices constitute a basis for $\ker M$
$$
     N_1=4(\mathbf{1}-c)b^T,\ N_{i+1}=P_{i-1}(c)b^T(P_s'(C)-P_{s-1}'(C)),\quad i=1,\ldots,s.
$$
 \end{lemma}
\begin{proof}
The linear independence of these matrices is evident. The identity
 $x(P_s'-P_{s-1}')=s(P_s+P_{s-1})$ shows that $\langle G_{q}, P_s'-P_{s-1}' \rangle_D=0$
 for all $q\geq 1$ so that \eqref{eq:kercond} is trivially satisfied for all $p$ and $q$.
 Since there are $s+1$ linarly independent elements, Lemma~\ref{lemma:rank10} ensures
 the rank of $M$ is precisely $s^2-s-1$ and 
  that the given set of matrices  forms a basis for $\ker M$ when $\zeta=-1$.
\end{proof}
 At this point we know that a necessary condition for a  Runge-Kutta method to be energy preserving for polynomial Hamiltonians is that its Butcher matrix is of the form $A=cb^T + N$ where $N\in\ker M$. 
Now, we make use of the row sum condition~\eqref{eq:rowsum} to infer the additional requirement on 
$N$ that $N\cdot\mathbf{1}=0$. The following lemma is easily proved.
\begin{lemma} \label{lemma:onedim}
For all $\zeta\in\mathbf{R}$, the intersection of $\ker M$ and the set of all $s\times s$ matrices
with zero row sum, i.e. $N\cdot\mathbf{1}=0$,
 is a one dimensional subspace of $\mathbf{R}^{s\times s}$ consisting of matrices of rank at most one.
\end{lemma}
%
%
From the preceding lemma, we now conclude that any  Butcher matrix we
search for is of the form $cb^T + \beta \,N$ where $N=Ub^TV\in\ker M$. We show that such
candidates are indeed incompatible with the nonlinear triple bush conditions presented  
in subsection~\ref{subsec:nonlincond}
unless $\beta=0$. Of particular use to us is the
  condition obtained by choosing $P(x)=Q(x)=G_p(x)$ and $R(x)=G_1(x)=x$ in \eqref{eq:3bush}. Because of the symmetry, and because $G_p(1)=\delta_{p1}$, we get the following simple special case:
\begin{equation} \label{eq:3bushp2q2}
2b^TP_{p-1}(C)ACAG_p(c)
-2\delta_{p1}b^TCAG_p(c)
-b^T(AG_p(c) \odot AG_p(c)) + \nu_p^2 = 0,
\end{equation} 
 where $\nu_1=\frac12$, $\nu_2=\frac1{6}$ and $\nu_p=0$ when $p\geq 2$.

\begin{lemma} \label{lemma:conclusion_generalcase}
If  $A=cb^T+\beta \,N$ is a solution to the triple bush condition \eqref{eq:3bushp2q2}
where $N=U(c)b^TV(C)\in\ker M$ with $\zeta\not\in\{-1,0\}$,   then $\beta=0$.
\end{lemma}

\begin{proof}
We substitute such a solution into \eqref{eq:3bushp2q2} with $p=1$, and $p=2$. The terms 
coming from $cb^T$  cancel since the AVF method is energy preserving. The linear terms in $\beta$
are proportional to $u_2v_0-u_1v_1$ and therefore vanish by \eqref{eq:kercond} using $p=1, q=2$.
The quadratic terms are computed by means of \eqref{eq:legortho} and \eqref{eq:biortho}, and we get the condition
$$
      -\frac{\beta^2}{(2p-1)^2} \langle 2u_p\,x\,V+ v_{p-1}U, v_{p-1}U\rangle_D = 0,\quad p=1,2.
$$
 We invoke \eqref{eq:uk}, \eqref{eq:us} to substitute for $v_{p-1} U$, 
 $$
      -\frac{\beta^2}{(2p-1)^2}\,u_p^2\, \langle 2xV+\bar{V}, \bar{V}\rangle_D  = 0,\quad\mbox{where}\quad
      \bar{V} = \sum_{k=1}^s v_{k-1} P_{k-1} - \zeta v_s P_{s-1}
$$
The inner product can be worked out, and one finally gets the condition
$$
     \frac{\beta^2}{(2p-1)^2}\,u_p^2\,v_s^2(\zeta+1)^2 = 0,\quad p=1,2.
$$
From Lemma~\ref{lemma:kerelts} and Lemma~\ref{lemma:rankkerelts} we deduce that
$u_1=0$ and $u_2=0$ simultaneously is impossible for any rank one kernel element.
For all $\zeta\neq 0$, $v_s=1$, so  the lemma holds as stated.
\end{proof}
We now consider the case $\zeta=0$ which corresponds to the Gauss-Legendre
 quadrature formula. From \eqref{eq:GLv2v3} we see that the second kernel element 
 has $V^2(x)=P_2'(x)$ which means that 
 $U^2(c)b^TV^2(C)\mathbf{1}=U^2(c)\langle P_2',1\rangle_D=0$, this causes
 the one dimensional subspace of kernel elements satisfying the row sum condition to be
 the span of
\begin{equation}
    N =   U(c)b^TV(C)=(P_0(c)-P_2(c)) b^TP_2'(C)
\end{equation}
where we have skipped the superscripts on $U$ and $V$ for ease of notation.
We use the condition \eqref{eq:asymbush} with $q=s$ which reads, after inserting
$A=cb^T+\beta N$, and using $\langle U,1\rangle_D=\langle V,x\rangle_D=1$,
\begin{equation} \label{eq:asymbushqs}
\langle (\tfrac{1}{2} x + \beta U)^s,1\rangle_D
+\tfrac{s}2\langle x, (\tfrac12 x + \beta U)^{s-1}\rangle_D
-\beta\langle xV, (\tfrac12 x + \beta U)^{s-1}\rangle_D
-\left(\tfrac12\right)^s = 0.
\end{equation}
We observe that these inner products involve polynomials of degree $2s$ whereas the underlying quadrature formula for this case has order $2s$ and is therefore exact for all polynomials of degree at most $2s-1$. Clearly, the condition \eqref{eq:asymbushqs} would hold exactly for all $\beta$ if all discrete inner products were replaced by continuous ones. We therefore conclude that it is only necessary to retain terms arising from the leading order $2s$, these are the terms multiplying $\beta^s$.
This results in the condition
$$
      (-1)^{s-1}\frac{(6\beta)^s}{\gamma_s^2} = 0
$$
and we have proved that any solution of the form $A=cb^T + \beta N$ requires $\beta=0$.
 
The final exceptional case is $\zeta=-1$ which corresponds to the Radau I quadrature and $c_1=0$. The one dimensional subspace of elements in $\ker M$ which satisfy the row sum condition is in this case given as the span of the matrix
$$
      N=U(c)b^TV(C),\quad      U(x)=P_0(x)-P_1(x),\quad  
      V(x)=(-1)^s P_1'(x) - P_{s-1}'(x) + P_s'(x)
$$
 We substitute $A=cb^T + \beta N$ into the triple bush condition \eqref{eq:3bush} with $P(x)=Q(x)=xG_2(x)$, $R(x)=1$.
 After evaluating all inner products, we find that  the condition yields
\begin{equation}
\label{eq:EPtriplebush122short}
-\frac49 \beta^2 = 0.
 \end{equation}
and therefore we must have $\beta=0$ and the only possible energy preserving method with
$\zeta=-1$ is the AVF method, $A=cb^T$.
We summarize these findings
\begin{lemma} \label{lemma:conclusion_exceptions}
Let $\zeta\in\{-1,0\}$,  $N\in\ker M$ and consider the matrix $A=cb^T + \beta N$, where
$N\cdot\mathbf{1}=0$. For the Runge-Kutta method with coefficients $(c,b,A)$ to be energy preserving, one must have $\beta=0$.
\end{lemma}

\begin{proof} Theorem~\ref{theo:main} (odd case). By Lemma~\ref{lemma:conclusion_generalcase} and
Lemma~\ref{lemma:conclusion_exceptions} the theorem is proved.

\end{proof}

\section*{Acknowledgments.}
The first two authors would like to acknowledge the support from the IRSES project CRISP, and part of the work was carried out while the authors were visiting  Massey University, Palmerston North, New Zealand and Latrobe University, Melbourne, Australia.

 \bibliographystyle{plain}
 \bibliography{mybib}

\end{document}